\documentclass[final,leqno]{siamltex}

\usepackage{amsfonts,amssymb} 
\usepackage{amsmath} 
\usepackage{color} 
\usepackage{graphicx} 
\usepackage{epsfig,mathrsfs} 
\usepackage[bf,SL,BF]{subfigure} 
\usepackage{fancyhdr} 
\usepackage{CJK} 
\usepackage{caption} 
\usepackage{wrapfig} 
\usepackage{cases} 
\usepackage{bm}
\usepackage{algorithm}
\usepackage{algorithmic}
\newtheorem{remark}{Remark}[section] 
\newtheorem{example}{Example}[section] 

\title{High order algorithm for  the time-tempered fractional Feynman-Kac equation
\thanks{This work was supported by NSFC 11671182. The first author was also partially  supported by the Fundamental Research Funds for the Central Universities
under Grant No. lzujbky-2016-105. }}

\author{Minghua Chen\thanks{Corresponding author. School of Mathematics and Statistics, Gansu Key Laboratory of Applied Mathematics and Complex Systems,
 Lanzhou University, Lanzhou 730000, P.R. China  (Email: chenmh@lzu.edu.cn).}
        \and Weihua Deng\thanks{ School of Mathematics and Statistics, Gansu Key Laboratory of Applied Mathematics and Complex Systems,
 Lanzhou University, Lanzhou 730000, P.R. China (Email: dengwh@lzu.edu.cn). }}

\begin{document}

\maketitle

\begin{abstract}
We provide and analyze the high order algorithms for the model describing the functional distributions of particles performing anomalous motion with power-law  jump length and tempered power-law
waiting time. The model is derived in [Wu, Deng, and Barkai, Phys. Rev. E., 84 (2016), 032151], being called the time-tempered fractional  Feynman-Kac equation.
The key step of designing the algorithms is to discretize the time tempered fractional substantial derivative, being defined as
$${^S\!}D_t^{\gamma,\widetilde{\lambda}} G(x,p,t)\!=\!D_t^{\gamma,\widetilde{\lambda}} G(x,p,t)\!-\!\lambda^\gamma G(x,p,t)
~{\rm with}~\widetilde{\lambda}=\lambda+ pU(x),\, p=\rho+J\eta,\, J=\sqrt{-1},$$
where
$$D_t^{\gamma,\widetilde{\lambda}} G(x,p,t) =\frac{1}{\Gamma(1-\gamma)} \left[\frac{\partial}{\partial t}+\widetilde{\lambda} \right]
\int_{0}^t{\left(t-z\right)^{-\gamma}}e^{-\widetilde{\lambda}\cdot(t-z)}{G(x,p,z)}dz,$$
and $\lambda \ge 0$, $0<\gamma<1$, $\rho>0$, and $\eta$ is a real number.
The designed schemes are unconditionally stable and have the global truncation error $\mathcal{O}(\tau^2+h^2)$, being theoretically proved and numerically verified
in {\em complex} space. Moreover, some simulations for the distributions of the first passage time are performed, and the second order convergence is also obtained for solving the `physical' equation (without artificial source term).
\end{abstract}



\pagestyle{myheadings}
\thispagestyle{plain}
\markboth{M. H. CHEN AND W. H. DENG}{ALGORITHM FOR TIME-TEMPERED FRACTIONAL FEYNMAN-KAC EQUATION}

\section{Introduction}

As a generalization of Brownian walk,  the continuous time random walk (CTRW) model allows the incorporation of the waiting time distribution $\psi(t)$
and the general jump length distribution $\eta(x)$ \cite{Klafter:11}. If both the first moment of $\psi(t)$ and the second moment of $\eta(x)$ are bounded, the CTRW model describes the normal diffusion, otherwise it characterizes the anomalous diffusion. Considering the life of the biological particles is finite, sometimes it is a more sensible choice to use the tempered power-law waiting time distribution, $e^{-\lambda t} t^{-\alpha-1}$, instead of the divergent power-law one, $t^{-\alpha-1}$; on the other hand, the new model can also describes the dynamics of very slow transition from subdiffusion to normal diffusion, then to superdiffusion, which has many potential applications in physical, biological, and chemical process
 \cite{Baeumera:10,Bruno:04,Cartea:07,del-Castillo-Negrete:09,Hanert:14,Stanislavsky:14,Wu:16}. In fact, there are several ways to make the tempering, e.g., directly truncating the heavy tail of the power-law distribution \cite{Sokolov:04}. Here we use the exponential tempering, which has both mathematical and practical advantages \cite{Sabzikar:14}.

The model we discuss in this paper, given in \cite{Wu:16}, is for the distribution of tempered non-Brownian functional. If $x(\tau)$ represents a tempered non-Brownian motion, a functional over a fixed time interval $[0,t]$ is defined as $A=\int_0^t U(x(\tau))d\tau$, where $U(x)$ is some prescribed function. For each realization of the tempered non-Brownian path, the quantity $A$ has a different value and one is interested in the probability density function of $A$, governed by the partial differential equation (PDE) derived in \cite{Wu:16}.  As listed in \cite{Majumdar:05}, the functional has many applications ranging from probability theory, finance, data analysis, and disordered systems to computer science. In probability theory, one of the important objects of interest is the occupation time $A=\int_0^t \theta (x(\tau))$, where $\theta(x)=1$ if $x>0$, otherwise $\theta(x)=0$. For the Kardar-Parisi-Zhang (KPZ) varieties, $U(x)$ is taken as $x^2$. In finance, for the integrated stock price
up to some `target' time $t$, $U(x)=e^{-\beta x}$. For describing the stochastic behaviour of daily temperature records, $U(x)$ is taken as $x \theta(x)$. The functional with $U(x)=|x|$ is studied by physicists in the context of electron-electron and phase coherence in one-dimensional weakly disordered quantum wire. One of the most important functionals is the first passage functional, which appears in physics, astronomy, queuing theory, etc. With the first passage time $t_f$, the functional over the time interval $[0,t_f]$ is defined as $A=\int_0^{t_f} U(x(\tau))d\tau$. In this paper, based on the designed algorithm, we will also simulate the distribution of the first passage tempered non-Brownian functional.


For the numerical methods of fractional PDEs, almost all of them are for fractional diffusion equations and the related PDEs; and some important progresses have been made, including the finite difference method \cite{Chen:0013,Li:14,Meerschaert:04,Tian:12,Zhang:12,Zhuang:09}, finite element method \cite{Deng:08,Ervin:06,Zeng:013},  spectral method \cite{Xu:09,YZX:14}, fast method \cite{Pang:12,Wang:12}, etc; more recently, by weighting and shifting Gr\"{u}nwald's first order discretization \cite{Meerschaert:04} and Lubich's discretization \cite{Lubich:86}, a series of effective high order discretizations for space/time fractional derivatives are derived \cite{Chen:0013,Hao:15,Ji:15,Li:14,Tian:12}. Still discussing the diffusion of particles but with the tempered power-law waiting time and/or jump length distributions, the tempered fractional derivative is firstly introduced in \cite{Cartea:07} and further applied and numerically solved in \cite{Baeumera:10,Hanert:14,Li:16,Sabzikar:14}. Based on an extension of the concept of CTRW to position-velocity space, Friedrich et al \cite{Friedrich:06} derive the fractional Kramers-Fokker-Planck equation, which involves a collision operator called fractional substantial derivative, representing important nonlocal couplings in time and space. After that, it is further  found that the fractional substantial derivative plays a core role in the fractional Feynman-Kac equation \cite{Carmi:11}, which governs the probability amplitude associated with an entire motion of a particle as a function of time, rather than simply with a position of the particle at particular time. The mathematical properties of the operator are detailedly discussed in \cite{Chen:13} and the fractional Feynman-Kac equation is numerically solved in \cite{CD:15}.

In the recently derived model \cite{Wu:16} for the tempered functional distribution, a so-called tempered fractional substantial derivative is introduced, being defined as
\begin{equation}\label{TemperedOperator}
{^S\!}D_t^{\gamma,\widetilde{\lambda}} G(x,p,t)\!=\!D_t^{\gamma,\widetilde{\lambda}} G(x,p,t)\!-\!\lambda^\gamma G(x,p,t)
~{\rm with}~\widetilde{\lambda}=\lambda+ pU(x),\, p=\rho+J\eta,
\end{equation}
where
\begin{equation}\label{SubstantialOperator}
D_t^{\gamma,\widetilde{\lambda}} G(x,p,t) =\frac{1}{\Gamma(1-\gamma)} \left[\frac{\partial}{\partial t}+\widetilde{\lambda} \right]
\int_{0}^t{\left(t-z\right)^{-\gamma}}e^{-\widetilde{\lambda}\cdot(t-z)}{G(x,p,z)}dz,
\end{equation}
and $\lambda \ge 0$, $0<\gamma<1$, $ J=\sqrt{-1}$, $\rho>0$, and $\eta$ is a real number. And the model, being the equivalent form of \cite[Eq. (32)]{Wu:16}, is
\begin{equation}\label{ModelEquation}
{^S\!}D_t^{\gamma,\widetilde{\lambda}} G(x,p,t) -D_t^{\gamma,\widetilde{\lambda}}\left[e^{-\widetilde{\lambda} t }G(x,p,0)\right]=K {\nabla}_x^{\alpha} G(x,p,t)+f(x,p,t),
\end{equation}
where $f(x,p,t)=\lambda D_t^{\gamma-1, \widetilde{\lambda}} e^{-pU(x)t } -\lambda^\gamma e^{-pU(x)t }$ and
$$
D_t^{\gamma-1, \widetilde{\lambda}} e^{-pU(x)t } =\frac{1}{\Gamma(1-\gamma)}
\int_{0}^t{\left(t-z\right)^{-\gamma}}e^{-\widetilde{\lambda}\cdot(t-z)}e^{-pU(x)z }dz.
$$
 This paper focuses on providing the high order discretization for the tempered fractional substantial derivative (\ref{TemperedOperator}). Then derive the high order algorithm for the time-tempered fractional Feynman-Kac equation and make the rigorous stability and convergence analysis to the algorithm. As a concrete application, we perform some numerical simulations for the distribution of the first passage time.

The rest of the paper is organized as follows. In the next section, we propose the high order discretization to the operator and high order algorithm to the model. In Sec. 3, we do the detailedly theoretical analyses for the stability and convergence with the second order accuracy in both time and space directions in complex space for the derived schemes. To verify the theoretical results, especially the convergence orders, the extensive numerical experiments are performed in Sec. 4; as an important application of the time-tempered fractional Feynman-Kac equation, we also numerically solve the distribution of the first passage time. The paper is concluded with some remarks in the last section.

\section{High order schemes for the time-tempered fractional Feynman-Kac equation}
The original time-tempered fractional Feynman-Kac equation \cite[Eq. (32)]{Wu:16} is derived as
\begin{equation}\label{OriginalModel}
\begin{split}
\frac{\partial}{\partial t}G_{x_0}(p,t)&=\left[ \lambda^\gamma D_t^{1-\gamma,\widetilde{\lambda}}-\lambda \right]\left[G_{x_0}(p,t)-e^{-pU (x_0)t} \right]-pU (x_0)G_{x_0}(p,t)
\\
&\quad
 +KD_t^{1-\gamma,\widetilde{\lambda}} {\nabla}_{x_0}^{\alpha} G_{x_0}(p,t),
\end{split}
\end{equation}
where the operator $D_t^{1-\gamma,\widetilde{\lambda}}$ is defined in (\ref{SubstantialOperator}). Following the ideas given in \cite{Deng:14,Deng:16}, rearranging the terms of (\ref{OriginalModel}), denoting $G_{x_0}(p,t)$ as $G(x,p,t)$, and performing $D_t^{\gamma-1,\widetilde{\lambda}}$ on both sides of (\ref{OriginalModel}) lead to (\ref{ModelEquation}). Like defining the Caputo fractional substantial derivative \cite{Chen:13,Deng:14}, for $\gamma \in (0,1)$, if we define the Caputo tempered fractional substantial derivative as
\begin{equation} \label{CaputoSubDerivative}
{_C^S}D_t^{\gamma,\widetilde{\lambda}}  G(x,p,t)=D_t^{\gamma,\widetilde{\lambda}} \left[G(x,p,t)-e^{-\widetilde{\lambda} t }G(x,p,0)\right]\!-\!\lambda^\gamma G(x,p,t),
\end{equation}
then (\ref{ModelEquation}) with appropriate boundary and initial conditions can be rewritten as
\begin{equation} \label{1.1}
\left\{ \begin{array}{lll}
\displaystyle
{_C^S}D_t^{\gamma,\widetilde{\lambda}}  G(x,p,t)= K {\nabla}_x^{\alpha} G(x,p,t)+f(x,p,t); & &\\
G(x,p,0)=G_0(x,p)& {\rm for} & x \in \Omega=(0,b);\\
G(x,p,t)=0  & {\rm for} & (x,t) \in ({\mathbb{R}} \backslash \Omega) \times (0,T],
\end{array}
 \right.
\end{equation}
where $f(x,p,t)$ is the same as the one given in (\ref{ModelEquation}), but in the following, we take it as a more general known function; and the Riesz fractional derivative with $\alpha \in (1,2)$,
is defined as  \cite{Podlubny:99,Tarasov:10}
\begin{equation}\label{1.2}
{\nabla}_x^{\alpha} G(x,p,t) =-\kappa_{\alpha}\left( _{-\infty}D_x^{\alpha}+ _{x}\!D_{+\infty}^{\alpha} \right)G(x,p,t) ~~{\rm with}~~\kappa_{\alpha}=\frac{1}{2\cos(\alpha \pi/2)},
\end{equation}
\begin{equation}\label{1.3}
 _{-\infty}D_x^{\alpha}G(x,p,t)=
\frac{1}{\Gamma(2-\alpha)} \displaystyle \frac{\partial^2}{\partial x^2}
 \int_{-\infty}\nolimits^x{\left(x-\xi\right)^{1-\alpha}}{G(\xi,p,t)}d\xi,
\end{equation}
\begin{equation}\label{1.4}
 _{x}D_{\infty}^{\alpha}G(x,p,t)=
 \frac{1}{\Gamma(2-\alpha)}\frac{\partial^2}{\partial x^2}
\int_{x}\nolimits^{\infty}{\left(\xi-x\right)^{1-\alpha}}{G(\xi,p,t)}d\xi.
\end{equation}
{\em Note:} Here ${\nabla}_x^{\alpha}G(x,p,t)=-\kappa_{\alpha}\left( _{0}D_x^{\alpha}+ _{x}\!D_{b}^{\alpha} \right)G(x,p,t)$, since we request $G(x,p,t)=0$ for $x \in (\mathbb{R} \backslash \Omega)$.

For designing the numerical scheme of (\ref{1.1}), we take the mesh points $x_i=ih,i=0,1,\ldots ,M$, and $t_n=n\tau,n=0,1,\ldots, N$, where
 $h=b/M$, $\tau=T/N$ are  the uniform space stepsize and time steplength, respectively.
Denote $G_{i,p}^n$ as the numerical approximation to $G(x_i,p,t_n)$.
In Subsection \ref{SubSec:2.1}, we provide the high order discretization schemes for the tempered fractional substantial derivative.

\subsection{Discretizations of the tempered fractional substantial derivative} \label{SubSec:2.1}
For numerically solving (\ref{1.1}), the first step is to discretize the tempered fractional substantial derivative, also defined in (\ref{TemperedOperator}),
\begin{equation}\label{2.1}
 {^S\!}D_t^{\gamma,\widetilde{\lambda}} G(x,p,t)\!=\!D_t^{\gamma,\widetilde{\lambda}} G(x,p,t)\!-\!\lambda^\gamma G(x,p,t),
\end{equation}
where
\begin{equation*}
0<\gamma<1,\,\widetilde{\lambda}=\lambda+ pU(x),~~ p=\rho+J\eta~~~{\rm with}~ \lambda\geq 0,\,  \rho>0,\, U(x)\ge 0, J=\sqrt{-1}.
\end{equation*}
\begin{remark}
The tempered fractional substantial derivative (\ref{2.1}) reduces to the fractional substantial derivative \cite{Chen:13,Friedrich:06} if $\lambda=0$.
The tempered fractional substantial derivative becomes to the tempered fractional  derivative \cite{Baeumera:10,Cartea:07,Li:16}  if $p=0$.
When $\lambda=0$ and $p=0$,  the operator  (\ref{2.1}) reduces to the traditional fractional derivative \cite{Tarasov:10}.
\end{remark}

Now, we introduce the following lemmas and denote
$${_{-\infty}}D_t^{\gamma,\widetilde{\lambda}} G(x,p,t) =\frac{1}{\Gamma(1-\gamma)} \left[\frac{\partial}{\partial t}+\widetilde{\lambda} \right]
\int_{-\infty}^t{\left(t-z\right)^{-\gamma}}e^{-\widetilde{\lambda}\cdot(t-z)}{G(x,p,z)}dz,$$
and
$${_{-\infty}\!\!\!}^SD_t^{\gamma,\widetilde{\lambda}} G(x,p,t)={_{-\infty}}D_t^{\gamma,\widetilde{\lambda}} G(x,p,t)-\lambda^\gamma G(x,p,t).$$
\begin{lemma}[\cite{Cartea:07,CD:15}]\label{lemma2.1}
 Let $0<\gamma<1$, $G(t)$ and ${_{-\infty}\!}D_t^{\gamma,\widetilde{\lambda}} G(t)$ belong to $L_1(\mathbb{R})$. Then
\begin{equation*}
\mathcal{F}\left({_{-\infty}}D_t^{\gamma,\widetilde{\lambda}} G(t) \right)
 =(\widetilde{\lambda}-i\omega)^{\gamma}\widehat{G}(\omega);
\end{equation*}
and
\begin{equation*}
\mathcal{F}\left({_{-\infty}\!\!\!}^SD_t^{\gamma,\widetilde{\lambda}} G(t) \right)
 =(\widetilde{\lambda}-i\omega)^{\gamma}\widehat{G}(\omega)-\lambda^{\gamma}\widehat{G}(\omega),
\end{equation*}
where $\mathcal{F}$ denotes Fourier transform operator and $\widehat{G}(\omega)=\mathcal{F}(G)$, i.e.,
\begin{equation*}
    \widehat{G}(\omega)=\int_{\mathbb{R}}e^{i\omega t }G(t)dt.
 \end{equation*}
\end{lemma}

In the following, using Fourier transform methods, we derive the $\nu$-th order ($\nu\leq 4$)  approximations of
the  tempered fractional substantial derivative  by the corresponding coefficients of the generating functions
 $\kappa ^{\nu,\gamma,\widetilde{\lambda}}(\zeta)$ with
\begin{equation}\label{2.2}
\kappa^{\nu,\gamma,\widetilde{\lambda}}(\zeta) = \left(\sum_{l=1}^\nu\frac{1}{l}\left(1- e^{-\widetilde{\lambda} \tau}\zeta  \right)^l\right)^{\gamma}
= \sum_{m=0}^{\infty}{q}_m^{\nu,\gamma,\widetilde{\lambda}}\zeta^m,~~~~\nu=1,2,3,4,
\end{equation}
where $\tau$ is the uniform time stepsize and
the coefficients \cite{Chen:13}:
\begin{equation} \label{2.3}
q_m^{\nu,\gamma,\widetilde{\lambda}}=e^{-\widetilde{\lambda} m\tau} l_m^{\nu,\gamma}
=e^{-\left[\lambda+\left(\rho+J\eta\right) U(x_i)  \right] m\tau} l_m^{\nu,\gamma},~ J=\sqrt{-1},~\nu=1,2,3,4.
\end{equation}
Here
\begin{equation}\label{2.4}
\begin{split}
{l}_m^{1,\gamma}&=\left(1-\frac{\gamma+1}{m}\right)l_{m-1}^{1,\gamma}~~{\rm with}~~  {l}_0^{1,\gamma}=1,\\
 {l}_m^{2,\gamma}
 & =\left(\frac{3}{2}\right)^{\gamma} \sum_{k=0}^{m} 3^{-k}\,{l}_k^{1,\gamma}\,{l}_{m-k}^{1,\gamma},
\end{split}
\end{equation}
 and the coefficients ${l}_m^{3,\gamma}$ and ${l}_m^{4,\gamma}$ are, respectively,  defined by (2.6) and (2.8)  of \cite{Chen:1313}.

\begin{lemma}\label{lemma2.2}
Let $G $, ${_{-\infty}\!\!\!}^SD_t^{\gamma+1,\widetilde{\lambda}}  G(t)$ and their Fourier transforms belong to $L_1(\mathbb{R})$ and
 \begin{equation}\label{2.5}
A^{1,\gamma,\widetilde{\lambda}}G(t)=\frac{1}{\tau^{\gamma}}\sum_{m=0}^{\infty}{q}_m^{1,\gamma,\widetilde{\lambda}}G(t-m\tau)
                            -\frac{1}{\tau^{\gamma}} \left(1-e^{-\lambda \tau}\right)^{\gamma} G(t).
\end{equation}
 Then
 $$  {_{-\infty}\!\!\!}^SD_t^{\gamma,\widetilde{\lambda}}  G(t) =A^{1,\gamma,\widetilde{\lambda}}G(t)+\mathcal{O}(\tau).$$
\end{lemma}
\begin{proof}
Using Fourier transform, we obtain
\begin{equation*}
\begin{split}
\mathcal{F}(A^{1,\gamma,\widetilde{\lambda}}G)(\omega)
&=\frac{1}{\tau^{\gamma}} \sum_{m=0}^{\infty}{q}_m^{1,\gamma,\widetilde{\lambda}} \left(e^{i\omega \tau}\right)^m \widehat{G}(\omega)-\frac{1}{\tau^{\gamma}} \left(1-\frac{1}{e^{\lambda \tau}}\right)^{\gamma}\widehat{G}(\omega)\\
&=\frac{1}{\tau^{\gamma}} \left(1-\frac{e^{i\omega \tau}}{e^{\widetilde{\lambda} \tau}}\right)^{\gamma}\widehat{G}(\omega)-\frac{1}{\tau^{\gamma}} \left(1-\frac{1}{e^{\lambda \tau}}\right)^{\gamma} \widehat{G}(\omega)\\ 
&=(\widetilde{\lambda}-i\omega)^{\gamma}\varphi(\widetilde{\lambda}-i\omega)\widehat{G}(\omega)-\lambda^\gamma \varphi(\lambda)\widehat{G}(\omega)
  \end{split}
\end{equation*}
with
\begin{equation}\label{2.6}
\begin{split}
 \varphi(z)&=\left(\frac{1-e^{-z\tau}}{z\tau}\right )^{\gamma}
 =1- \frac{\gamma}{2}z\tau +\frac{3\gamma^2+\gamma}{24} z^2 \tau^2-\frac{\gamma^3+\gamma^2}{48} z^3\tau ^3 +\mathcal{O}(\tau^4).
\end{split}
\end{equation}
Therefore, from Lemma \ref{lemma2.1},  there exists
\begin{equation*}
\begin{split}
\mathcal{F}(A^{1,\gamma,\widetilde{\lambda}}G)(\omega)=\mathcal{F}( {_{-\infty}\!\!\!}^SD_t^{\gamma,\widetilde{\lambda}}G)+ \widehat{\phi}(\omega),
  \end{split}
\end{equation*}
where $ \widehat{\phi}(\omega)
=-\frac{\gamma}{2}\left[(\widetilde{\lambda} -i\omega)^{\gamma+1}-\lambda^{\gamma+1}\right]\widehat{G}(\omega)\cdot \tau+\mathcal{O}(\tau^2)$.  Then
\begin{equation*}
\begin{split}
&|\widehat{\phi}(\omega)| \leq c\cdot\left(|(\widetilde{\lambda} -i\omega)^{\gamma+1}\widehat{G}(\omega)|+|\lambda^{\gamma+1}\widehat{G}(\omega)|\right)\cdot \tau.
\end{split}
\end{equation*}
With the conditions $\widehat{G}(\omega)$ and $\mathcal{F}[{_{-\infty}\!\!\!}^SD_t^{\gamma+1,\widetilde{\lambda}}   G(t)] \in L_1(\mathbb{R})$, it leads to
\begin{equation*}
\begin{split}
| {_{-\infty}\!\!\!}^SD_t^{\gamma,\widetilde{\lambda}}   G(t)-A^{1,\gamma,\widetilde{\lambda}}G(t)|=|\phi(t)| \leq \frac{1}{2\pi}\int_{\mathbb{R}}|\widehat{\phi}(\omega)|dt
= \mathcal{O}(\tau).
  \end{split}
\end{equation*}
The proof is completed.
\end{proof}
\begin{lemma}\label{lemma2.3}
Let $G $, ${_{-\infty}\!\!\!}^SD_t^{\gamma+2,\widetilde{\lambda}} G(t)$ and their Fourier transforms belong to $L_1(\mathbb{R})$ and
 \begin{equation}\label{2.7}
A^{2,\gamma,\widetilde{\lambda}}G(t)=\frac{1}{\tau^{\gamma}}\sum_{m=0}^{\infty}{q}_m^{2,\gamma,\widetilde{\lambda}}G(t-m\tau)
                            -\frac{1}{\tau^\gamma}(1-e^{-\lambda \tau})^\gamma \left[1+\frac{1}{2}\left(1-e^{-\lambda \tau}\right)\right]^\gamma G(t).
\end{equation}
 Then
 $$  {_{-\infty}\!\!\!}^SD_t^{\gamma,\widetilde{\lambda}}  G(t) =A^{2,\gamma,\widetilde{\lambda}}G(t)+\mathcal{O}(\tau^2).$$
\end{lemma}
\begin{proof}
Using Fourier transform method, we have
\begin{equation*}
\begin{split}
\mathcal{F}(A^{2,\gamma,\widetilde{\lambda}}G)(\omega)
&=\frac{1}{\tau^\gamma}\left(1-e^{-(\widetilde{\lambda} -i\omega) \tau}\right)^\gamma \left[1+\frac{1}{2}\left(1-e^{-(\widetilde{\lambda }-i\omega) \tau}\right)\right]^\gamma \widehat{G}(\omega)\\
&\quad -\frac{1}{\tau^\gamma}(1-e^{-\lambda \tau})^\gamma \left[1+\frac{1}{2}\left(1-e^{-\lambda \tau}\right)\right]^\gamma \widehat{G}(\omega)\\
&=(\widetilde{\lambda}-i\omega)^{\gamma}\varphi(\widetilde{\lambda}-i\omega) \psi(\widetilde{\lambda}-i\omega)\widehat{G}(\omega)-\lambda^\gamma \varphi(\lambda)\psi(\lambda)\widehat{G}(\omega),
  \end{split}
\end{equation*}
where the function  $\varphi(z)$ is defined by (\ref{2.6}) and
 \begin{equation}\label{2.8}
\begin{split}
\psi(z)
&=\left( 1+\frac{1}{2}\left(1-e^{-z\tau}\right)\right)^{\gamma}\\
&=1+ \frac{\gamma}{2}z\tau+\frac{\gamma(\gamma-3)}{8}z^2\tau^2+\frac{\gamma(\gamma^2-9\gamma+12)}{48}z^3\tau^3+\mathcal{O}(\tau^4).
  \end{split}
\end{equation}
According to (\ref{2.6}) and (\ref{2.8}), there exists
\begin{equation}\label{2.9}
\begin{split}
\varphi(z)\psi(z)
&= \left(\frac{1-e^{-z\tau}}{z\tau}\right )^{\gamma} \!\! \left( 1+\frac{1}{2}\left(1-e^{-z\tau}\right)\right)^{\gamma}\\
&\quad  =1-\frac{\gamma}{3}z^2\tau^2+\frac{\gamma}{4}z^3\tau^3+\mathcal{O}(z^4).
  \end{split}
\end{equation}
Again, using Lemma \ref{lemma2.1} and (\ref{2.9}),  we obtain
\begin{equation*}
\begin{split}
\mathcal{F}(A^{2,\gamma,\widetilde{\lambda}}G)(\omega)=\mathcal{F}(  {_{-\infty}\!\!\!}^SD_t^{\gamma,\widetilde{\lambda}}   G)+ \widehat{\phi}(\omega),
  \end{split}
\end{equation*}
where $ \widehat{\phi}(\omega)
=-\frac{\gamma}{3}\left[(\widetilde{\lambda} -i\omega)^{\gamma+2}-\lambda^{\gamma+2}\right]\widehat{G}(\omega)\cdot \tau^2+\mathcal{O}(\tau^3)$.  Then
\begin{equation*}
\begin{split}
&|\widehat{\phi}(\omega)| \leq c\cdot\left(|(\widetilde{\lambda} -i\omega)^{\gamma+2}\widehat{G}(\omega)|+|\lambda^{\gamma+2}\widehat{G}(\omega)|\right)\cdot \tau^2,
\end{split}
\end{equation*}
which leads to
\begin{equation*}
\begin{split}
| {_{-\infty}\!\!\!}^SD_t^{\gamma,\widetilde{\lambda}}G(t)-A^{2,\gamma,\widetilde{\lambda}}G(t)|=|\phi(t)| \leq \frac{1}{2\pi}\int_{\mathbb{R}}|\widehat{\phi}(\omega)|dt= \mathcal{O}(\tau^2).
  \end{split}
\end{equation*}
The proof is completed.
\end{proof}

By the similar arguments performed in Lemma \ref{lemma2.3}, we further have the following results.
\begin{lemma}\label{lemma2.4}
Let $G $, ${_{-\infty}\!\!\!}^SD_t^{\gamma+p,\widetilde{\lambda}}G(t)$ ($p=1,2,3,4$)  and their Fourier transforms belong to $L_1(\mathbb{R})$ and
 \begin{equation}\label{2.10}
A^{p,\gamma,\widetilde{\lambda}}G(t)=\frac{1}{\tau^{\gamma}}\sum_{m=0}^{\infty}{q}_m^{p,\gamma,\widetilde{\lambda}}G(t-m\tau)
                            -\frac{1}{\tau^\gamma}\left(\sum_{l=1}^\nu\frac{1}{l}\left(1- e^{-\lambda \tau} \right)^l\right)^{\gamma} G(t).
\end{equation}
 Then
 $${_{-\infty}\!\!\!}^SD_t^{\gamma,\widetilde{\lambda}} G(t) =A^{p,\gamma,\widetilde{\lambda}}G(t)+\mathcal{O}(\tau^\nu),~~~~\nu=1,2,3,4.$$
\end{lemma}

\subsection{Derivation of the numerical schemes}

Nowadays, there are already several types of high
order discretization schemes for the Riemann-Liouville space fractional derivatives \cite{Chen:0013,Hao:15,Ortigueira:06,Sousa:12,Tadjeran:06,Tian:12}.
Here, we take the following schemes \cite{Tian:12} to approximate (\ref{1.3}) and  (\ref{1.4}).
\begin{equation}\label{2.11}
 \begin{split}
_{0}D_{x}^{\alpha}G(x)|_{x=x_i}
&=\frac{1}{h^{\alpha}}\delta_{x,+}^\alpha G(x_i) +\mathcal{O}(h^2)~~{\rm with}~~\delta_{x,+}^\alpha G(x_i)=\sum_{k=0}^{i+1}w_{k}^{\alpha}G(x_{i-k+1}),\\
_{x}D_{b}^{\alpha}G(x)|_{x=x_i}
&=\frac{1}{h^{\alpha}}\delta_{x,-}^\alpha G(x_i) +\mathcal{O}(h^2)~~{\rm with}~~ \delta_{x,-}^\alpha G(x_i)=\sum_{k=0}^{M-i+1}w_{k}^{\alpha}G(x_{i+k-1}),
 \end{split}
\end{equation}
where
\begin{equation}\label{2.12}
 \begin{split}
w_{0}^{\alpha}=\frac{\alpha}{2}g_{0}^{\alpha},~~w_{k}^{\alpha}=\frac{\alpha}{2}g_{k}^{\alpha}+\frac{2-\alpha}{2}g_{k-1}^{\alpha},~~k\geq 1,
 \end{split}
\end{equation}
and
\begin{equation}\label{2.13}
g_k^{\alpha}=(-1)^k\left ( \begin{matrix}\alpha \\ k\end{matrix} \right ),~~{\rm i.e.,}~~
g_0^{\alpha}=1, ~~~~g_k^{\alpha}=\left(1-\frac{\alpha+1}{k}\right)g_{k-1}^{\alpha},~~k \geq 1.
\end{equation}

According to  (\ref{1.2}) and (\ref{2.11}), we obtain the approximation operator of the Riesz fractional derivative
\begin{equation}\label{2.14}
\begin{split}
 \nabla_x^{\alpha} G(x)|_{x=x_i}
 &=-\frac{\kappa_{\alpha}}{h^\alpha}\left(\delta_{x,+}^\alpha +\delta_{x,-}^\alpha  \right)G(x_{i}) +\mathcal{O}(h^2)\\
&=-\frac{\kappa_{\alpha}}{h^\alpha} \sum_{j=0}^{M}w_{i,j}^{\alpha}G(x_{j})   +\mathcal{O}(h^2),
\end{split}
\end{equation}
where  $i=1,\ldots,M-1$ (together with the zero Dirichlet boundary conditions)  and
\begin{equation}\label{2.15}
w_{i,j}^{\alpha}=\left\{ \begin{array}
 {l@{\quad } l}
  w_{i-j+1}^{\alpha},&j < i-1,\\
  w_{0}^{\alpha}+w_{2}^{\alpha} ,&j=i-1,\\
 2w_{1}^{\alpha},&j=i,\\
w_{0}^{\alpha}+w_{2}^{\alpha} ,&j=i+1,\\
w_{j-i+1}^{\alpha} ,&j>i+1.
 \end{array}
 \right.
\end{equation}
Taking $G=[G({x_1}),G({x_2}),\cdots,G({x_{M-1}})]^{\rm T}$, and  using (\ref{2.11}), (\ref{2.14}), there exists
 \begin{equation}\label{2.16}
\begin{split}
& \frac{1}{h^\alpha}\left[\sum_{j=0}^{M}w_{1,j}^{\alpha}G(x_{j}),\sum_{j=0}^{M}w_{2,j}^{\alpha}G(x_{j}),\ldots,\sum_{j=0}^{M}w_{M-1 ,j}^{\alpha}G(x_{j}) \right]^T\\
&= \frac{1}{h^\alpha}\left(\delta_{x,+}^\alpha +\delta_{x,-}^\alpha  \right)G   =  \frac{1}{h^\alpha}A_\alpha G,
\end{split}
\end{equation}
where the matrix
\begin{equation}\label{2.17}
A_\alpha=B_{\alpha}+B_{\alpha}^T~~{\rm with} ~~
B_\alpha=\left [ \begin{matrix}
w_1^{\alpha}   &w_2^{\alpha}&w_3^{\alpha}     &      \cdots   &  w_{M\!-\!2}^{\alpha}     &  w_{M\!-\!1}^{\alpha}  \\
w_0^{\alpha}&  w_1^{\alpha}    &w_2^{\alpha}&  w_3^{\alpha}    &     \cdots   &  w_{M\!-\!2}^{\alpha} \\
            &w_0^{\alpha}&w_1^{\alpha}         & w_2^{\alpha}&     \ddots              & \vdots  \\
                   &                  &       \ddots            &        \ddots            &      \ddots             &  w_3^{\alpha}  \\
  &          &              &        \ddots            &   w_1^{\alpha}         &w_2^{\alpha} \\
   &      &          &                   &w_0^{\alpha}& w_1^{\alpha}
 \end{matrix}
 \right ].
\end{equation}

According to Lemma \ref{lemma2.4} and (\ref{2.14}), we can write (\ref{1.1}) as
\begin{equation}\label{2.18}
\begin{split}
& \frac{1}{\tau^\gamma}\sum_{k=0}^{n}{q}_k^{\nu,\gamma,\widetilde{\lambda}}G(x_i,p,t_{n-k})
 -\frac{1}{\tau^\gamma}\left(\sum_{l=1}^\nu\frac{1}{l}\left(1- e^{-\lambda \tau} \right)^l\right)^{\gamma} G(x_i,p,t_n)\\
 &\quad  -\frac{1}{\tau^\gamma}\sum_{k=0}^{n}{q}_k^{\nu,\gamma,\widetilde{\lambda}}e^{-\widetilde{\lambda} (n-k)\tau}G(x_i,p,0)\\
& =-\frac{K\kappa_\alpha}{h^\alpha}
\sum_{j=0}^{M}w_{i,j}^{\alpha}G(x_j,p,t_n)+f(x_i,p,t_n)+ r_i^{n}
\end{split}
\end{equation}
with the local truncation error
\begin{equation}\label{2.19}
  |r_i^n| \leq C_G(\tau^\nu+h^2),~~\nu=1,2,3,4,
\end{equation}
where $C_G$ is a positive constant independent of $\tau$ and $h$.

From (\ref{2.3}) and (\ref{2.18}),  the resulting discretization of (\ref{1.1}) can be rewritten as
\begin{equation}\label{2.20}
\begin{split}
& \frac{1}{\tau^\gamma} \sum_{k=0}^{n}e^{-\left[\lambda+\rho U(x_i)  \right] k\tau} l_k^{\nu,\gamma}  \left[ e^{-J\eta U(x_i)  k\tau} G_{i,p}^{n-k}\right]
  -\frac{1}{\tau^\gamma}\left(\sum_{l=1}^\nu\frac{1}{l}\left(1- e^{-\lambda \tau} \right)^l\right)^{\gamma}G_{i,p}^{n} \\
  &  = \kappa  \frac{1}{h^\alpha}\sum_{j=0}^{M}w_{i,j}^{\alpha}G_{j,p}^{n}+\frac{1}{\tau^\gamma}\sum_{k=0}^{n}l_k^{\nu,\gamma} e^{-\widetilde{\lambda} n\tau}G_{i,p}^{0}+ f_{i,p}^{n},~~J=\sqrt{-1}
\end{split}
\end{equation}
with $\kappa=-K\kappa_\alpha>0$.

When $\nu=2$, the finite difference schemes (\ref{2.20}) reduces to
\begin{equation}\label{2.21}
\begin{split}
& \frac{1}{\tau^\gamma} \sum_{k=0}^{n}l_k^{(2)}  e^{-J\eta U(x_i)  k\tau} G_{i,p}^{n-k}
  -\kappa  \frac{1}{h^\alpha}\sum_{j=0}^{M}w_{i,j}^{\alpha}G_{j,p}^{n}
  = \frac{1}{\tau^\gamma}\sum_{k=0}^{n}l_k^{2,\gamma} e^{-\widetilde{\lambda} n\tau}G_{i,p}^{0}+ f_{i,p}^{n},
\end{split}
\end{equation}
where
\begin{equation}\label{2.22}
\begin{split}
l_0^{(2)}= l_0^{2,\gamma} -\left(\sum_{l=1}^2\frac{1}{l}\left(1- e^{-\lambda \tau} \right)^l\right)^{\gamma},~~~
l_k^{(2)}=e^{-\left[\lambda+\rho U(x_i)  \right] k\tau} l_k^{2,\gamma}~~ \forall k\geq 1.
\end{split}
\end{equation}
Without loss of generality, supposing (\ref{1.1}) with zero initial value \cite{Ji:15}, then (\ref{2.21}) reduces to the simple scheme
\begin{equation}\label{2.23}
\begin{split}
& \frac{1}{\tau^\gamma} \sum_{k=0}^{n}l_k^{(2)}  e^{-J\eta U(x_i)  k\tau} G_{i,p}^{n-k}
  -\kappa  \frac{1}{h^\alpha}\sum_{j=0}^{M}w_{i,j}^{\alpha}G_{j,p}^{n}
  =  f_{i,p}^{n}.
\end{split}
\end{equation}

\section{Convergence and Stability Analysis}
In this section, we  theoretically prove  that the above designed scheme is unconditionally stable  and  2nd order convergent in both of space and time directions.

\subsection{A few technical lemmas}
First, we introduce some relevant notations and properties of discretized inner product. Denote  $u^n=\{u_i^n| 0 \leq i \leq M, n \geq 0 \}$
and $v^n=\{v_i^n| 0 \leq i \leq M, n \geq 0 \}$, which are grid functions. And
\begin{equation*}
\begin{split}
&(u^n,v^n)=h\sum_{i=1}^{M-1}u_i^n\overline{v}_i^n, ~~~~~||u^n||=(u^n,u^n)^{1/2}.
\end{split}
\end{equation*}

\begin{lemma}[\cite{Meerschaert:04,Meerschaert:06,Tian:12}] \label{lemma3.1}
 Let $1<\alpha<2$. Then the coefficients in  (\ref{2.13}) and (\ref{2.12})  satisfy
\begin{equation} \label{3.1}
\left\{ \begin{array}
 {l@{\quad} l}
\displaystyle g_0^\alpha>0,~ g_1^\alpha<0, ~g_k^\alpha>0~~\forall k\geq 2,\\
\\
\displaystyle \sum_{k=0}^\infty g_k^\alpha=0, ~~\sum_{k=0}^m g_k^\alpha<0,~~ m\geq 1;
\end{array}
 \right.
\end{equation}
and
\begin{equation} \label{3.2}
\left\{ \begin{array}
 {l@{\quad} l}
\displaystyle w_0^\alpha>0,~ w_1^\alpha<0,~ w_0^\alpha+w_2^\alpha>0, ~w_k^\alpha>0~~\forall k\geq 3,\\
\\
\displaystyle \sum_{k=0}^\infty w_k^\alpha=0, ~~\sum_{k=0}^m w_k^\alpha<0, ~~ m\geq 1.
\end{array}
 \right.
\end{equation}
\end{lemma}

\begin{lemma}\label{lemma3.2}
 Let $0<\gamma<1$ and  $1<\alpha<2$. Then the coefficients $g^\gamma_k$ and $g^\alpha_k$  in (\ref{2.13}), respectively,  satisfy
\begin{equation}\label{3.3}
\begin{split}
& g^\gamma_0=1; ~~~~ g^\gamma_k<0,~~ (k\geq 1); ~~~~\sum_{k=0}^{n-1}g^\gamma_k>0;  ~~~~\sum_{k=0}^{\infty}g^\gamma_k=0;\\
 & 0<\frac{1}{n^\gamma \Gamma(1-\gamma)}< \sum_{k=0}^{n-1}g^\gamma_k=-\sum_{k=n}^{\infty}g^\gamma_k \leq \frac{1}{n^\gamma} ~~{\rm for}~0<\gamma<1,~~n\geq 1.
  \end{split}
\end{equation}
And
\begin{equation}\label{3.4}
\begin{split}
& - \frac{1}{n^\alpha} < \sum_{k=0}^{n}g^\alpha_k=-\sum_{k=n+1}^{\infty}g^\alpha_k \leq \frac{1}{n^\alpha \Gamma(1-\alpha)}<0 ~~{\rm for}~1<\alpha<2,~~n\geq 1.
  \end{split}
\end{equation}
\end{lemma}

\begin{proof}
 The first estimate (\ref{3.3}) of  this Lemma  can be seen in Lemma 2.1 of \cite{Deng:14}. Here, we mainly prove the second one. We can check that
 \begin{equation*}
   \sum_{k=0}^n g^\alpha_k=\sum_{k=0}^n (-1)^k  {\alpha \choose k}=\sum_{k=0}^n  {k-\alpha-1 \choose k}={n-\alpha \choose n},
 \end{equation*}
which leads to
 \begin{equation*}
  \frac{ \sum_{k=0}^n g^\alpha_k}{ \sum_{k=0}^{n-1} g^\alpha_k}=1-\frac{\alpha}{n}~~\forall n\geq 1.
 \end{equation*}

 For $x \in (0,1)$ and $1<\alpha<2$, there exists
  \begin{equation*}
  1-\alpha x<(1-x)^\alpha=1-\alpha x+\frac{\alpha(\alpha-1)}{2!}x^2-\frac{\alpha(\alpha-1)(\alpha-2)}{3!}x^3+\cdots,
 \end{equation*}
which implies that
  \begin{equation*}
  1- \frac{\alpha}{n}<\left(1-\frac{1}{n}\right)^\alpha=\frac{(n-1)^\alpha}{n^\alpha}.
 \end{equation*}
Therefore, we have
 \begin{equation*}
  \frac{ \sum_{k=0}^n g^\alpha_k}{ \sum_{k=0}^{n-1} g^\alpha_k}<\frac{(n-1)^\alpha}{n^\alpha},
 \end{equation*}
i.e.,
 \begin{equation}\label{3.5}
 (n-1)^\alpha \sum_{k=0}^{n-1} g_k^\alpha<n^\alpha  \sum_{k=0}^n g_k^\alpha.
 \end{equation}

From \cite{Chen:09,Podlubny:99}, we have
 \begin{equation*}
 \frac{x^{-\alpha}}{\Gamma(1-\alpha)}={_0}D_x^\alpha 1=\lim_{h\rightarrow0,nh=x} h^{-\alpha}\sum_{k=0}^n g^\alpha_k.
 \end{equation*}
For any fixed positive constant $x$, it yields
 \begin{equation}\label{3.6}
\lim_{n\rightarrow\infty} n^{\alpha}\sum_{k=0}^n g^\alpha_k= \frac{1}{\Gamma(1-\alpha)}<0~~{\rm with}~1<\alpha<2.
 \end{equation}
According to (\ref{3.5}) and (\ref{3.6}), we obtain
\begin{equation}\label{3.7}
\begin{split}
& \sum_{k=0}^{n}g^\alpha_k=-\sum_{k=n+1}^{\infty}g^\alpha_k \leq \frac{1}{n^\alpha \Gamma(1-\alpha)}<0 ~~{\rm for}~1<\alpha<2,~~n\geq 1.
  \end{split}
\end{equation}

Next we prove the following inequality by mathematical induction
\begin{equation}\label{3.8}
 - \frac{1}{n^\alpha}< \sum_{k=0}^{n}g^\alpha_k=-\sum_{k=n+1}^{\infty}g^\alpha_k~~{\rm for}~~n\geq 1.
\end{equation}
It is obvious that (\ref{3.8}) holds when $n=1,2$, since
$$g^\alpha_0+g^\alpha_1>-1~~{\rm and }~~\min\limits_{\alpha \in (1,2)} \left(g^\alpha_0+g^\alpha_1+g^\alpha_2\right)=-\frac{1}{4}>-\frac{1}{2^\alpha}.$$
 Supposing that
\begin{equation*}
  -\frac{1}{s^\alpha} \leq  \sum_{i=0}^{s}g^\alpha_i=-\sum_{i=s+1}^{\infty}g^\alpha_i, ~~s=1,2,\ldots, n-1,
\end{equation*}
and using (\ref{3.5}), we obtain
$$ \sum_{k=0}^{n}g^\alpha_k>\frac{(n-1)^\alpha}{n^\alpha} \sum_{k=0}^{n-1} g^\alpha_k\geq - \frac{1}{(n-1)^\alpha} \frac{(n-1)^\alpha}{n^\alpha}=-\frac{1}{n^\alpha}
~~{\rm for}~~n\geq 2, $$
which leads to
$$ \sum_{k=0}^{n}g^\alpha_k>-\frac{1}{n^\alpha}~~{\rm for}~~n\geq 1. $$
The proof is completed.
\end{proof}

\begin{lemma} [\cite{Chen:0013,Li:14}]\label{lemma3.3}
Let $0<\gamma<1$. Then the coefficients ${l}_m^{2,\gamma}$ defined in (\ref{2.4}) satisfy
 \begin{equation*}
 \begin{split}
l_0^{2,\gamma}&=\left(\frac{3}{2}\right)^{\gamma}>0;
~~~~~~~~~~~~~~~~~~~~~~ l_1^{2,\gamma}=-\left(\frac{3}{2}\right)^{\gamma}\frac{4\gamma}{3}<0; ~~\quad\\
l_2^{2,\gamma}&=\left(\frac{3}{2}\right)^{\gamma}\frac{\gamma(8\gamma-5)}{9};】
~~~~~~~~~~~~~~~l_3^{2,\gamma}=\left(\frac{3}{2}\right)^{\gamma}\frac{4\gamma(\gamma-1)(7-8\gamma)}{81}; \\
l_k^{2,\gamma}&<0~~\forall k \geq 4;
~~~~~~~~~~~~~~~~~~~~~~~~ \sum_{k=0}^{\infty}l_k^{2,\gamma}=0.\\
 \end{split}
\end{equation*}
 \end{lemma}

\begin{lemma} \label{lemma3.5}
 Let $1<\alpha<2$ and $A_\alpha $ be given in (\ref{2.17}). Then
\begin{equation*}
-  \frac{1}{h^\alpha}(A_\alpha u,u)\geq -\frac{2}{b^\alpha\Gamma(1-\alpha)}||u||^2>0 ~~{\rm with}~u\in \mathbb{C}^{M-1},~\Omega=(0,b).
\end{equation*}
\end{lemma}
\begin{proof}
Let the vector $u=(u_1,u_2,\ldots,u_{M-1})^T$ with $u_0=u_M=0$. Since $A_\alpha$ is symmetric, we just need to the case that $u$ is a real vector.

From (\ref{2.17}) and Lemma \ref{lemma3.1}, there exists
\begin{equation*}
\begin{split}
&(B_\alpha u,u)\\
&=h\sum_{i=1}^{M-1} \left(\sum_{k=0}^{M-i}w^\alpha_ku_{i+k-1}  \right)u_i
              =h\sum_{k=0}^{M-1}w^\alpha_k \left(\sum_{i=1}^{M-k}u_{i+k-1} u_i \right)\\
&=w^\alpha_1 h\sum_{i=1}^{M-1}u_i^2 +\left(w^\alpha_0+w^\alpha_2\right) h\sum_{i=1}^{M-2}u_iu_{i+1}+h\sum_{k=3}^{M-1}w^\alpha_k \left(\sum_{i=1}^{M-k}u_{i+k-1} u_i \right)\\
&\leq w^\alpha_1 ||u||^2 +\left(w^\alpha_0+w^\alpha_2\right) h\sum_{i=1}^{M-2}\frac{u_i^2+u_{i+1}^2}{2}+h\sum_{k=3}^{M-1}w^\alpha_k \left(\sum_{i=1}^{M-k}\frac{u_i^2+u_{i+k-1}^2}{2} \right)\\
& \leq \left(\sum_{k=0}^{M-1}w^\alpha_k \right)||u||^2=\left(\sum_{k=0}^{M-2}g^\alpha_k +\frac{\alpha}{2}g^\alpha_{M-1}\right)||u||^2\\
&\leq  \left(\sum_{k=0}^{M-1}g^\alpha_k \right)||u||^2\leq  \left(\sum_{k=0}^{M}g^\alpha_k \right)||u||^2.
\end{split}
\end{equation*}
Since $(B^T_\alpha u,u)=(B_\alpha u,u)$ and $A_\alpha=\frac{1}{2}(B^T_\alpha+B_\alpha)$, we have
\begin{equation}\label{3.16}
  (A_\alpha u,u)\leq  2\left(\sum_{k=0}^{M}g^\alpha_k \right)||u||^2.
\end{equation}

From (\ref{3.16}) and (\ref{3.4}), we obtain
\begin{equation*}
\begin{split}
-  \frac{1}{h^\alpha}(A_\alpha u,u)
&\geq -\frac{2}{h^\alpha} \left(\sum_{k=0}^{M}g^\alpha_k \right)||u||^2
\geq \frac{2}{h^\alpha}   \frac{-1}{M^\alpha\Gamma(1-\alpha)}||u||^2\\
&\geq -\frac{2}{b^\alpha\Gamma(1-\alpha)}||u||^2>0 ~~{\rm with}~~\Omega=(0,b),~~u \in \mathbb{R}^{M-1}.
\end{split}
\end{equation*}

The proof is completed.
\end{proof}

\begin{lemma}\label{lemma3.6}
Let $0<\gamma<1$ and  the coefficients ${l}_k^{2,\gamma}$ be  given  in (\ref{2.4}). Then
\begin{equation*}
\begin{split}
  g(\gamma,z)=\sum_{k=1}^{\infty}l^{2,\gamma}_k\left(\cos(kz)-1\right)\geq 0,~~z\in[0,\pi].
\end{split}
\end{equation*}
\end{lemma}
\begin{proof}
We prove the above result for three cases: $\gamma \in \left(0,5/8\right] $; $\gamma \in\left[5/8,7/8\right]$; and $\gamma \in \left[7/8, 1\right)$.

Case 1: $\gamma \in \left(0,5/8\right] $. From Lemma \ref{lemma3.3}, we have  $l_k^{2,\gamma}\leq 0$,  $\forall k \geq 1$, which yields
$$g(\gamma,z)\geq 0.$$

Case 2: $\gamma \in\left[5/8,7/8\right]$. Using Lemma \ref{lemma3.3}, there exists $l_1^{2,\gamma}<0$, $l_2^{2,\gamma}\geq 0$ and $l_k^{2,\gamma}\leq 0$ $\forall k \geq 3$.
If we can prove $\sum_{k=1}^{2}l^{2,\gamma}_k\left(\cos(kz)-1\right)\geq 0$, the result holds obviously. It can be easily checked that
\begin{equation*}
\begin{split}
  \sum_{k=1}^{2}l^{2,\gamma}_k\left(\cos(kz)-1\right)
  =\left(\frac{3}{2}\right)^{\gamma}\frac{2\gamma}{9}h(z)
\end{split}
\end{equation*}
with $h(z)= \left( 8\gamma-5 \right) \cos^2(z)-6\cos(z)+11-8\gamma .$  Since
$$h^{'}(z)=2\sin(z)\left[ 3-\left( 8\gamma-5 \right)\cos(z) \right]\geq 0,~~z\in[0,\pi],~~\gamma \in \left[5/8,7/8\right],$$
it yields  $h(z)\geq h(0)=0$.

Case 3: $\gamma \in \left[7/8, 1\right)$. According to Lemma \ref{lemma3.3}, it yields  $l_1^{2,\gamma}<0$, $l_2^{2,\gamma}\geq 0$, $l_3^{2,\gamma}\geq 0$ and $l_k^{2,\gamma}\leq 0$ $\forall k \geq 4$.
So, if $\sum_{k=1}^{3}l^{2,\gamma}_k\left(\cos(kz)-1\right)\geq 0$ holds, the result is obtained. It can be easily got that
\begin{equation*}
\begin{split}
  p(z):&=\sum_{k=1}^{3}l^{2,\gamma}_k\left(\cos(kz)-1\right)\\
 & =4l^{2,\gamma}_3\cos^3(z)+2l^{2,\gamma}_2\cos^2(z)+\left(l^{2,\gamma}_1-3l^{2,\gamma}_3\right)\cos(z) -\left( l^{2,\gamma}_1+2l^{2,\gamma}_2+l^{2,\gamma}_3\right).
\end{split}
\end{equation*}
For the above equation, we have
\begin{equation*}
\begin{split}
  p'(z)=-\sin(z)q(z)
\end{split}
\end{equation*}
with
\begin{equation*}
\begin{split}
  q(z) =12l^{2,\gamma}_3\cos^2(z)+4l^{2,\gamma}_2\cos(z)+\left(l^{2,\gamma}_1-3l^{2,\gamma}_3\right);
\end{split}
\end{equation*}
and
\begin{equation*}
\begin{split}
  q'(z) =-\sin(z)r(z)
\end{split}
\end{equation*}
with
\begin{equation*}
\begin{split}
  r(z) &=24l^{2,\gamma}_3\cos(z)+4l^{2,\gamma}_2\geq 4l^{2,\gamma}_2-24l^{2,\gamma}_3\\
 & =\left(\frac{3}{2}\right)^{\gamma}\frac{4\gamma}{9}\left[8\gamma-5-\frac{24}{9}(1-\gamma)\left(8\gamma-7\right)\right]\\
& \geq \left(\frac{3}{2}\right)^{\gamma}\frac{4\gamma}{9} \left( 2-\frac{1}{3}\right)>0~~~{\rm with}~\gamma \in \left[7/8, 1\right).
\end{split}
\end{equation*}
According to the above equations, we have $q'(z)\leq 0$,
$$q(z)\leq q(0)=\left(\frac{3}{2}\right)^{\gamma}\frac{4\gamma}{9}\left(1-\gamma\right)\left(8\gamma-15\right)<0,$$
and  $p'(z)\geq 0$,  $p(z)\geq p(0)=0.$
The proof is completed.
\end{proof}

\begin{lemma} \label{lemma3.7}
Let $l_k^{(2)}$ be defined by (\ref{2.22}) with $\lambda\geq 0$, $\rho U(x)\geq 0$, $J=\sqrt{-1}$.
Then for any positive integer $N$ and  complex vector $\left( v_i^0,v_i^1,\ldots,v_i^N \right) \in \mathbb{C}^{N+1}$, it holds that
$$\Re\left\{\sum_{n=0}^{N}\left(\sum_{k=0}^{n}l_k^{(2)}  e^{-J\eta U(x_i)  k\tau} v_i^{n-k}\right)\overline{v_i^n}\right\}\geq 0,~~i=1,2,\ldots, M-1,$$
where $\Re$ denotes the real part of the quantity.
\end{lemma}
\begin{proof}
 By the mathematical induction method, we can prove that
\begin{equation}\label{3.17}
\Re\left\{\sum_{n=0}^{N}\left(\sum_{k=0}^{n}l_k^{(2)}  e^{-J\eta U(x_i)  k\tau} v_i^{n-k}\right)\overline{v_i^n}\right\}=V_iL^{(2)}\overline{V_i}^T, ~~i=1,2,\ldots, M-1,
\end{equation}
 where
 $$V_i=\left( e^{-J\eta U(x_i)  N\tau } v_i^0,  e^{-J\eta U(x_i)  (N-1)\tau }v_i^1,\ldots,  e^{-J\eta U(x_i)  \tau }v_i^{N-1}, e^{-J\eta U(x_i) \cdot 0 \cdot \tau }v_i^N \right),$$
  and the real symmetric matrix
 \begin{equation}\label{3.18}
L^{(2)}=\left [ \begin{matrix}
l^{(2)}_0  &\frac{l^{(2)}_1}{2}&\frac{l^{(2)}_2}{2}    &      \cdots   & \frac{l^{(2)}_{N-1}}{2}     &  \frac{l^{(2)}_N}{2}  \\
\frac{l^{(2)}_1}{2}&  l^{(2)}_0   &\frac{l^{(2)}_1}{2}&  \frac{l^{(2)}_2}{2}   &     \cdots   & \frac{l^{(2)}_{N-1}}{2} \\
\frac{l^{(2)}_2}{2}            &\frac{l^{(2)}_1}{2}&l^{(2)}_0        & \frac{l^{(2)}_1}{2}&     \ddots              & \vdots  \\
\vdots                   &          \ddots         &       \ddots            &        \ddots            &      \ddots             &  \frac{l^{(2)}_2}{2} \\
  \frac{l^{(2)}_{N-1}}{2}  &   \ddots         &       \ddots        &        \ddots            &   l^{(2)}_0        & \frac{l^{(2)}_1}{2} \\
\frac{l^{(2)}_N}{2}   &   \frac{l^{(2)}_{N-1}}{2}   &   \cdots          &         \cdots           &\frac{l^{(2)}_1}{2}& l^{(2)}_0
 \end{matrix}
 \right ].
\end{equation}
Next we prove that the real symmetric  matrix  $L^{(2)}$ defined in (\ref{3.18}) is positive semi-definite.

We know that  the generating function \cite[p.12-14]{Chan:07} of $L^{(2)}$ is
\begin{equation}\label{3.19}
\begin{split}
 f(\gamma,z)
  &=\frac{1}{2}\sum_{k=0}^{\infty}l^{(2)}_ke^{Jkz}+\frac{1}{2}\sum_{k=0}^{\infty}l^{(2)}_ke^{-Jkz}=\sum_{k=0}^{\infty}l^{(2)}_k\cos(kz)\\
  &=\sum_{k=0}^{\infty}e^{-\left[\lambda+\rho U(x_i)  \right] k\tau} l_k^{2,\gamma}\cos(kz) -\left(\sum_{l=1}^2\frac{1}{l}\left(1- e^{-\lambda \tau} \right)^l\right)^{\gamma}\\
  &=\sum_{k=0}^{\infty}e^{-\left[\lambda+\rho U(x_i)  \right] k\tau} l_k^{2,\gamma}\cos(kz) -\sum_{k=0}^{\infty}e^{-\lambda k\tau} l_k^{2,\gamma}\\
  &=\sum_{k=1}^{\infty}e^{-\left[\lambda+\rho U(x_i)  \right] k \tau}l^{2,\gamma}_k\left(\cos(kz)-e^{\rho U(x_i)k\tau}\right),~~z\in [0,\pi].\\
\end{split}
\end{equation}
Since $ f(\gamma,z)$  is an even function and $2\pi$-periodic continuous real-valued functions defined on $[-\pi,\pi]$, we just need to consider its principal value on $[0,\pi]$.
Next we prove that $f(\gamma,z)$ defined in (\ref{3.19}) is nonnegative.

Case a: $\gamma \in \left(0,5/8\right] $. According to  Lemma \ref{lemma3.3}, we know that $e^{-\left[\lambda+\rho U(x_i)  \right] k \tau}l^{2,\gamma}_k\leq 0$   $\forall k \geq 1$, which yields
$f(\gamma,z)\geq 0.$

Case b: $\gamma \in\left[5/8,7/8\right]$.  It can be easily checked that
\begin{equation*}
\begin{split}
 &\sum_{k=1}^{2}e^{-\left[\lambda+\rho U(x_i)  \right] k \tau}l^{2,\gamma}_k\left(\cos(kz)-e^{\rho U(x_i)k\tau}\right)\\
&= \sum_{k=1}^{2}e^{-\left[\lambda+\rho U(x_i)  \right] k \tau}l^{2,\gamma}_k\left(\cos(kz)-1\right)
+\sum_{k=1}^{2}e^{-\left[\lambda+\rho U(x_i)  \right] k \tau}l^{2,\gamma}_k\left(1-e^{\rho U(x_i)k\tau}\right).
\end{split}
\end{equation*}
Using Case 2 of Lemma {\ref{lemma3.6}} and $l_1^{2,\gamma}<0$, $l_2^{2,\gamma}\geq 0$, there exists
\begin{equation*}
\begin{split}
 & \sum_{k=1}^{2}e^{-\left[\lambda+\rho U(x_i)  \right] k \tau}l^{2,\gamma}_k\left(\cos(kz)-1\right)\\
 & =e^{-2(\lambda+\rho U(x_i))  \tau}\left[ e^{(\lambda+\rho U(x_i))  \tau}l^{2,\gamma}_1\left(\cos(z)-1\right)+l^{2,\gamma}_2\left(\cos(2z)-1\right)  \right]\\
 &\geq  e^{-2(\lambda+\rho U(x_i))  \tau} \sum_{k=1}^{2}l^{2,\gamma}_k\left(\cos(kz)-1\right)\geq 0.
\end{split}
\end{equation*}
On the other hand, we have
\begin{equation*}
\begin{split}
&\sum_{k=1}^{2}e^{-\left[\lambda+\rho U(x_i)  \right] k \tau}l^{2,\gamma}_k\left(1-e^{\rho U(x_i)k\tau}\right)\\
&=e^{-(\lambda+\rho U(x_i))  \tau}\left(1-e^{\rho U(x_i)\tau}\right)\left(l^{2,\gamma}_1+e^{-(\lambda+\rho U(x_i))  \tau}l^{2,\gamma}_2\left(1+e^{\rho U(x_i)\tau}\right) \right)\\
&=e^{-(\lambda+\rho U(x_i))  \tau}\left(1-e^{\rho U(x_i)\tau}\right)\left(l^{2,\gamma}_1+e^{-\lambda  \tau}l^{2,\gamma}_2\left(e^{-\rho U(x_i)\tau}+1\right) \right)\\
&\geq e^{-(\lambda+\rho U(x_i))  \tau}\left(1-e^{\rho U(x_i)\tau}\right)\left(l^{2,\gamma}_1+2l^{2,\gamma}_2 \right)\geq 0,
\end{split}
\end{equation*}
since $l^{2,\gamma}_1+2l^{2,\gamma}_2=\left(\frac{3}{2}\right)^{\gamma}\frac{2\gamma}{9}\left(8\gamma-11\right) <0$.
Then we obtain  $f(\gamma,z)\geq 0.$

Case c: $\gamma \in \left[7/8, 1\right)$. It can be easily got that
\begin{equation*}
\begin{split}
 &\sum_{k=1}^{3}e^{-\left[\lambda+\rho U(x_i)  \right] k \tau}l^{2,\gamma}_k\left(\cos(kz)-e^{\rho U(x_i)k\tau}\right)\\
&= \sum_{k=1}^{3}e^{-\left[\lambda+\rho U(x_i)  \right] k \tau}l^{2,\gamma}_k\left(\cos(kz)-1\right)
+\sum_{k=1}^{3}e^{-\left[\lambda+\rho U(x_i)  \right] k \tau}l^{2,\gamma}_k\left(1-e^{\rho U(x_i)k\tau}\right).
\end{split}
\end{equation*}
From Case 3 of Lemma {\ref{lemma3.6}} and   $l_1^{2,\gamma}<0$, $l_2^{2,\gamma}\geq 0$, $l_3^{2,\gamma}\geq 0$, we get
\begin{equation*}
\begin{split}
 & \sum_{k=1}^{3}e^{-\left[\lambda+\rho U(x_i)  \right] k \tau}l^{2,\gamma}_k\left(\cos(kz)-1\right)\\
& =e^{-2\left(\lambda+ \rho U(x_i)\right)  \tau}\Big[ e^{\left(\lambda+ \rho U(x_i)\right)  \tau}l^{2,\gamma}_1\left(\cos(z)-1\right)+l^{2,\gamma}_2\left(\cos(2z)-1\right) \\
&\quad +  e^{-\left(\lambda+ \rho U(x_i)\right) \tau}l^{2,\gamma}_3\left(\cos(3z)-1\right)\Big]
\geq  e^{-2\left(\lambda+ \rho U(x_i)\right)  \tau} \sum_{k=1}^{3}l^{2,\gamma}_k\left(\cos(kz)-1\right)\geq 0.
\end{split}
\end{equation*}
On the other hand,
\begin{equation*}
\begin{split}
&\sum_{k=1}^{3}e^{-\left[\lambda+\rho U(x_i)  \right] k \tau}l^{2,\gamma}_k\left(1-e^{\rho U(x_i)k\tau}\right)\\
&=e^{-(\lambda+\rho U(x_i))  \tau}\left(1-e^{\rho U(x_i)\tau}\right)\Big[l^{2,\gamma}_1+e^{-\lambda  \tau}l^{2,\gamma}_2\left(e^{-\rho U(x_i)\tau}+1\right)\\
&\quad +e^{-\lambda  2\tau}l^{2,\gamma}_3\left(e^{-\rho U(x_i)2\tau}+e^{-\rho U(x_i)\tau}+1\right) \Big]\\
&\geq e^{-(\lambda+\rho U(x_i))  \tau}\left(1-e^{\rho U(x_i)\tau}\right)\left(l^{2,\gamma}_1+2l^{2,\gamma}_2 +3l^{2,\gamma}_3\right)\geq 0,
\end{split}
\end{equation*}
where $l^{2,\gamma}_1+2l^{2,\gamma}_2 +3l^{2,\gamma}_3=\left(\frac{3}{2}\right)^{\gamma}\frac{2\gamma}{27}\left(-16\gamma^2+54\gamma-47\right) <0$,  $\gamma \in \left[7/8, 1\right)$.

Using  $f(\gamma,z)\geq 0$ and Grenander-Szeg\"{o} theorem \cite[p.13-14]{Chan:07},
it implies that $L^{(2)}$  is a real symmetric positive semi-definite matrix. The proof is completed.
\end{proof}

\subsection{Convergence and stability analysis}

First, we give a priori estimate for simplifying the proof of the convergence and stability.
For the convenience, we analyze the numerical stability of the scheme with zero initial condition \cite{Deng:16,Ji:15}.


\begin{lemma}\label{lemma3.8}
Suppose $\{v_i^n\}$ is the solution of the difference scheme
\begin{equation}\label{3.20}
\begin{split}
& \frac{1}{\tau^\gamma} \sum_{k=0}^{n}l_k^{(2)}  e^{-J\eta U(x_i)  k\tau} v_{i}^{n-k}
  -\kappa  \frac{1}{h^\alpha}\sum_{j=0}^{M}w_{i,j}^{\alpha}v_{j}^{n}
  =  f_{i}^{n}, \\
&v_{i}^{0}=\phi_i, ~~1\leq i\leq M-1,\\
&v_0^n=v_M^n=0,~~0\leq n \leq N.
\end{split}
\end{equation}
Then for any positive integer $N$  with $N\tau\leq T$, it holds that
\begin{equation*}
\begin{split}
  \tau \sum_{n=1}^{N}||v^n||^2
&\leq    \frac{\tau^{1-\gamma}l_0^{(2)}b^\alpha \left|\Gamma(1-\alpha)\right|}{\kappa}||v^{0}||^2+  \frac{b^{2\alpha}\Gamma^2(1-\alpha)}{4\kappa^2} \cdot\tau\sum_{n=1}^{N}|| f^{n}||^2,
\end{split}
\end{equation*}
where $(x,t)\in (0,b) \times (0,T]$.
\end{lemma}

\begin{proof}
Multiplying (\ref{3.20}) by $h\overline{v_i^{n}}$ and summing up for $i$ from $1$ to $M-1$, we have
\begin{equation}\label{3.21}
\begin{split}
&\sum_{i=1}^{M-1}\frac{h}{\tau^\gamma} \left(\sum_{k=0}^{n}l_k^{(2)}  e^{-J\eta U(x_i)  k\tau} v_i^{n-k}\right)\overline{v_i^{n}}-  \kappa \frac{1}{h^\alpha}\left(A_\alpha v^{n},v^{n}\right) =\left(f^{n},v^{n}\right).
\end{split}
\end{equation}
Further multiplying (\ref{3.21})  by $\tau$ and summing up for $n$ from $1$ to $N$ and adding  $\tau^{1-\gamma}l_0^{(2)} ||v^{0}||^2$ on both sides of the obtained results, it yields
\begin{equation}\label{3.22}
\begin{split}
&\tau^{1-\gamma}h \sum_{i=1}^{M-1}\left\{\sum_{n=0}^{N} \left(\sum_{k=0}^{n}l_k^{(2)}  e^{-J\eta U(x_i)  k\tau} v_i^{n-k}\right)\overline{v_i^{n}}\right\}
 -   \frac{\kappa\tau}{h^\alpha}\sum_{n=1}^{N}\left(A_\alpha v^{n},v^{n}\right) \\
&= \tau\sum_{n=1}^{N}\left(f^{n},v^{n}\right)+\tau^{1-\gamma}l_0^{(2)} ||v^{0}||^2.
\end{split}
\end{equation}
Taking the real part on both sides of the  (\ref{3.22}) and using Lemmas  \ref{lemma3.5} and \ref{lemma3.7},
and the Schwarz inequality, Young's inequality,   we obtain
\begin{equation}\label{3.23}
\begin{split}
 \frac{-  2\kappa\tau}{b^\alpha \Gamma(1-\alpha)} \sum_{n=1}^{N}||v^n||^2
&\leq \tau\sum_{n=1}^{N}||f^{n}||\cdot||v^{n}||+\tau^{1-\gamma}l_0^{(2)} ||v^{0}||^2\\
&\leq \tau\sum_{n=1}^{N}\left( \epsilon ||v^{n}||^2+\frac{|| f^{n}||^2}{4\epsilon}\right) +\tau^{1-\gamma}l_0^{(2)} ||v^{0}||^2,
\end{split}
\end{equation}
where  $\epsilon>0$ and we use $\Re\left(f^{n},v^{n}\right)\leq |\left(f^{n},v^{n}\right)|\leq ||f^{n}||\cdot||v^{n}|| $.

Let  $\epsilon=\frac{-\kappa}{b^\alpha \Gamma(1-\alpha)}>0$. Using (\ref{3.23}), we obtain
\begin{equation*}
\begin{split}
& \left(\frac{-  2\kappa}{b^\alpha \Gamma(1-\alpha)}-\epsilon\right) \tau\sum_{n=1}^{N}||v^n||^2 \leq
 \frac{1}{4\epsilon}\tau\sum_{n=1}^{N}|| f^{n}||^2 +\tau^{1-\gamma}l_0^{(2)} ||v^{0}||^2,
\end{split}
\end{equation*}
which leads to
\begin{equation*}
\begin{split}
  \tau \sum_{n=1}^{N}||v^n||^2
&\leq  \frac{\tau^{1-\gamma}l_0^{(2)}}{\epsilon} ||v^{0}||^2+ \frac{1}{4\epsilon^2}\cdot\tau\sum_{n=1}^{N}|| f^{n}||^2\\
&=  \frac{\tau^{1-\gamma}l_0^{(2)}b^\alpha \left|\Gamma(1-\alpha)\right|}{\kappa}||v^{0}||^2+  \frac{b^{2\alpha}\Gamma^2(1-\alpha)}{4\kappa^2} \cdot\tau\sum_{n=1}^{N}|| f^{n}||^2.
\end{split}
\end{equation*}
The proof is completed.
\end{proof}

From the above lemma, we can obtain the following result.

\begin{theorem}\label{theorem3.9}
The difference scheme (\ref{2.23}) is unconditionally stable.
\end{theorem}

\begin{theorem}\label{theorem3.10}
 Let $G_{i,\rho}^n$ be the approximate solution of $G(x_i,\rho,t_n)$ computed by the difference scheme (\ref{2.21}).
 Denote $\varepsilon_i^n=G(x_i,\rho,t_n)-G_{i,\rho}^n$. Then
\begin{equation*}
\begin{split}
  \tau \sum_{n=1}^{N}||\varepsilon^n||
  &\leq  \frac{C_G|\Gamma(1-\alpha)| b^{\alpha+\frac{1}{2}}T}{2\kappa}   \cdot(\tau^2+h^2).
\end{split}
\end{equation*}
where $C_G$ is defined by  (\ref{2.19}) and  $(x_i,t_n)\in (0,b) \times (0,T]$ with $N\tau\leq T$.
\end{theorem}

\begin{proof}
Let $G(x_i,\rho,t_n)$ be the exact solution of (\ref{1.1}) at the mesh point $(x_i,t_n)$,
and $\varepsilon_i^n=G(x_i,\rho,t_n)-G_{i,\rho}^n$.
Subtracting (\ref{2.18}) from (\ref{2.21}) with $\varepsilon_i^0=0$, it yields
\begin{equation}\label{3.020}
\begin{split}
& \frac{1}{\tau^\gamma} \sum_{k=0}^{n}l_k^{(2)}  e^{-J\eta U(x_i)  k\tau} \varepsilon_{i}^{n-k}
  -\kappa  \frac{1}{h^\alpha}\sum_{j=0}^{M}w_{i,j}^{\alpha}\varepsilon_{j}^{n}
  =  r_{i}^{n};
\end{split}
\end{equation}
and  using Lemma \ref{lemma3.8}, it holds that
\begin{equation*}
\begin{split}
  \tau \sum_{n=1}^{N}||\varepsilon^n||^2
&\leq    \frac{b^{2\alpha}\Gamma^2(1-\alpha)}{4\kappa^2} \cdot\tau\sum_{n=1}^{N}|| r^{n}||^2
 \leq  C^2_G\frac{\Gamma^2(1-\alpha)b^{1+2\alpha}T}{4\kappa^2}   \cdot(\tau^2+h^2)^2.
\end{split}
\end{equation*}
Using the Cauchy-Schwarz inequality and the above inequality, it yields
\begin{equation*}
\begin{split}
  \left(\tau \sum_{n=1}^{N}||\varepsilon^n||\right)^2
  &\leq   \left(\tau \sum_{n=1}^{N}1\right) \left(\tau \sum_{n=1}^{N}||\varepsilon^n||^2\right)\\
  &\leq   C^2_G\frac{\Gamma^2(1-\alpha)b^{1+2\alpha}T^2}{4\kappa^2}   \cdot(\tau^2+h^2)^2,
\end{split}
\end{equation*}
which leads to
\begin{equation*}
\begin{split}
  \tau \sum_{n=1}^{N}||\varepsilon^n||
  &\leq  \frac{C_G|\Gamma(1-\alpha)| b^{\alpha+\frac{1}{2}}T}{2\kappa}   \cdot(\tau^2+h^2).
\end{split}
\end{equation*}
The proof is completed.
\end{proof}

\section{Numerical results}
We numerically verify the above theoretical results including convergence orders and numerical stability.  And the $ l_\infty$ norm is used to measure the numerical errors. We further extend the application of the algorithm to simulate the probability of the first passage time.

\subsection{Numerical results for $G(x,p,t)$}
In this subsection, we give the following  two examples: one is a  artificial solution and the other is a unknown solution for (\ref{1.1}).
\begin{example}\end{example}
Consider (\ref{1.1}) on a finite domain with $0<x <1$,  $0<t \leq 1$, the coefficient  $K=1$, and the forcing function
\begin{equation*}
\begin{split}
f(x,p,t)=&\frac{\Gamma(3+\gamma)}{\Gamma(3)}e^{-\left(\lambda+ \left(\rho+J\eta\right)x \right) t}\cdot t^2x^2(1-x)^2\\
&  -\lambda^\alpha e^{-\left(\lambda+ \left(\rho+J\eta\right)x \right) t}\cdot\left(t^{2+\gamma}+1\right)x^2(1-x)^2\\
       &  +\frac{e^{-\lambda t}(t^{2+\gamma}+1)}{2\cos(\alpha \pi/2)} \left({ _{0}}D_x^{\alpha}+{ _{x}}D_1^{\alpha}\right)[e^{-\left(\rho+J\eta\right)tx}\cdot x^2(1-x)^2],
\end{split}
\end{equation*}
where the left and right fractional derivatives of the given functions are calculated by the algorithm presented in the Appendixes of \cite{CD:15, Hesthaven:07}.
The initial condition is $G(x,p,0)=x^2(1-x)^2$ with the zero boundary
conditions. Then (\ref{1.1}) has the exact
solution $$G(x,p,t)=e^{-\left(\lambda+ \left(\rho+J\eta\right)U(x) \right) t}\left(t^{2+\gamma}+1\right)x^2(1-x)^2. $$

\begin{table}[h]\fontsize{9.5pt}{12pt}\selectfont
 \begin{center}
  \caption {The maximum errors and convergence orders for  (\ref{2.21}) with  $\nu=2$, $h=\tau,$ $U(x)=x$, $J=\sqrt{-1}$.} \vspace{5pt}
\begin{tabular*}{\linewidth}{@{\extracolsep{\fill}}*{8}{c}}                                    \hline  
$\left(\lambda,\rho,\eta\right)$ &$h$ & $\alpha=1.3,\gamma=0.8$  & Rate        & $\alpha=1.8,\gamma=0.3$ &   Rate   \\\hline
                        &        ~~1/20&  1.1304e-03   &             & 1.2148e-03  &           \\
($\frac{1}{5}$,1,5)     &        ~~1/40&  2.8014e-04   &  2.0126     & 2.9680e-04  & 2.0332     \\
                        &        ~~1/80&  6.9327e-05   &  2.0146     & 7.2854e-05  & 2.0264      \\
                        &        ~~1/160&  1.7138e-05  &  2.0162     & 1.7845e-05  & 2.0295       \\\hline 
                        &        ~~1/20&  8.0830e-05   &             & 7.6493e-05               \\
(3,1,5)                 &        ~~1/40&  2.0077e-05   &  2.0094     & 1.8679e-05  & 2.0339      \\
                        &        ~~1/80&  4.9428e-06   &   2.0221    & 4.5848e-06  & 2.0264       \\
                        &        ~~1/160&  1.2151e-06  &   2.0243    & 1.1222e-06  & 2.0305        \\\hline 
    \end{tabular*}\label{table:4.1}
  \end{center}
\end{table}


Tables  \ref{table:4.1}  shows that the schemes (\ref{2.21}) with $\nu=2$ have the global truncation errors $\mathcal{O} (\tau^2+h^2)$  at time $T=1$.
\begin{example}\label{example4}\end{example}
Consider (\ref{1.1}) on a finite domain with $0<x <2$,  $0<t \leq 4$, the coefficient  $K=1$.
The initial condition is $x^2(2-x)^2$ with the homogeneous boundary conditions and the forcing function is $f(x,p,t)=0$. The solution of the corresponding steady-state equation of (\ref{1.1})  belongs to $C^{\alpha/2}(\mathbb{R})$ \cite{Ros-Oton:14}.

Since the analytic solutions is unknown for Example \ref{example4}, the order of the convergence of the numerical results are computed by the following formula
\begin{equation*}
  {\rm Convergence ~Rate}=\frac{\ln \left(||G_{2h}^N-G_{h}^N||_\infty/||G_{h}^N-G_{h/2}^N||_\infty\right)}{\ln 2}.
\end{equation*}

\begin{table}[h]\fontsize{9.5pt}{12pt}\selectfont
 \begin{center}
  \caption {The maximum errors and convergence orders for  (\ref{2.21}) with  $\nu=2$, $\tau=\frac{1}{400},$ $U(x)=x$, $J=\sqrt{-1}$.} \vspace{5pt}
\begin{tabular*}{\linewidth}{@{\extracolsep{\fill}}*{8}{c}}                                    \hline  
$\left(\lambda,\rho,\eta\right)$ &$h$ & $\alpha=1.3,\gamma=0.8$  & Rate        & $\alpha=1.8,\gamma=0.3$ &   Rate   \\
~                                &~~  & $||G_{h}^N-G_{h/2}^N||_\infty$  & ~        & $||G_{h}^N-G_{h/2}^N||_\infty$ &      \\\hline
                        &        ~~1/10&  3.1725e-05   &             & 3.8133e-05  &           \\
($\frac{1}{5}$,1,5)     &        ~~1/20&  7.4247e-06   &  2.0952     & 8.1807e-06  & 2.2207     \\
                        &        ~~1/40&  1.8780e-06   &  1.9831     & 1.9971e-06  & 2.0343      \\
                        &        ~~1/80&  4.8609e-07   &  1.9499     & 5.0141e-07  & 1.9938      \\\hline 
    \end{tabular*}\label{table:4.3}
  \end{center}
\end{table}

Table \ref{table:4.3} shows that the scheme (\ref{2.21}) still preserves the desired second-order convergence for the `physical' equation (without the artificial source term).


\subsection{First passage time}

The first passage time has many applications in physics, astronomy, and queuing theory, which is defined as the time $t_f$ that takes a particle stating at $x=-c_0$ ($c_0>0$) to hit a $x=0$ for the first time.  Obviously, $t_f$ is a random variable being different from the fixed time $t$ used the definition of the functional $A$ given in the Introduction section. But we still can build the connection between the probability of $t_f>t$ and $G(x,p,t)$ \cite{Wu:16}, i.e.,
$$P_r\{ t_f>t\}=\lim_{p\rightarrow \infty}G(x,p,t).$$
Behaviors of $P_r\{ t_f>t\}$ are simulated by solving the time tempered fractional Feynman-Kac equations (\ref{1.1}) with the algorithm (\ref{2.21}) on a finite domain  $-100< x < 100 $,  $0<t \leq 800$  with the coefficient $K=1$ and
\begin{equation*}
U(x)=\left\{ \begin{array}
{l@{\quad} l}0,~~~~x \leq 0,\\
1,~~~~x  >0.
\end{array}
\right.
\end{equation*}
It can be easily seen that the reasonable initial condition should be taken as $G(x,p,0)=1$. And we take
 $p=10^8$ and $\tau=h=1/20$.
\begin{figure}[t]
    \begin{minipage}[t]{0.50\linewidth}
    \includegraphics[scale=0.45]{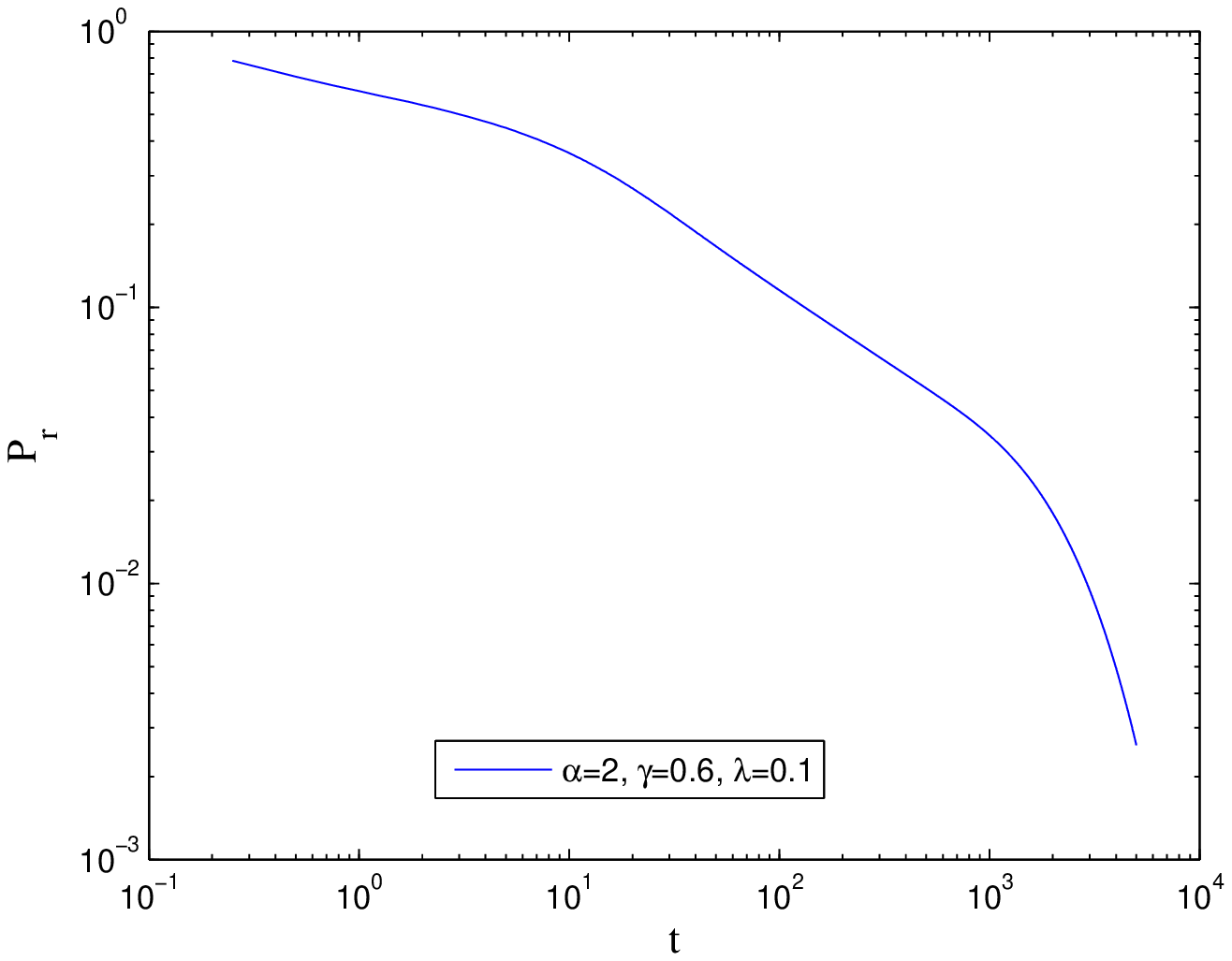}
    \end{minipage}
    \begin{minipage}[t]{0.50\linewidth}
    \includegraphics[scale=0.45]{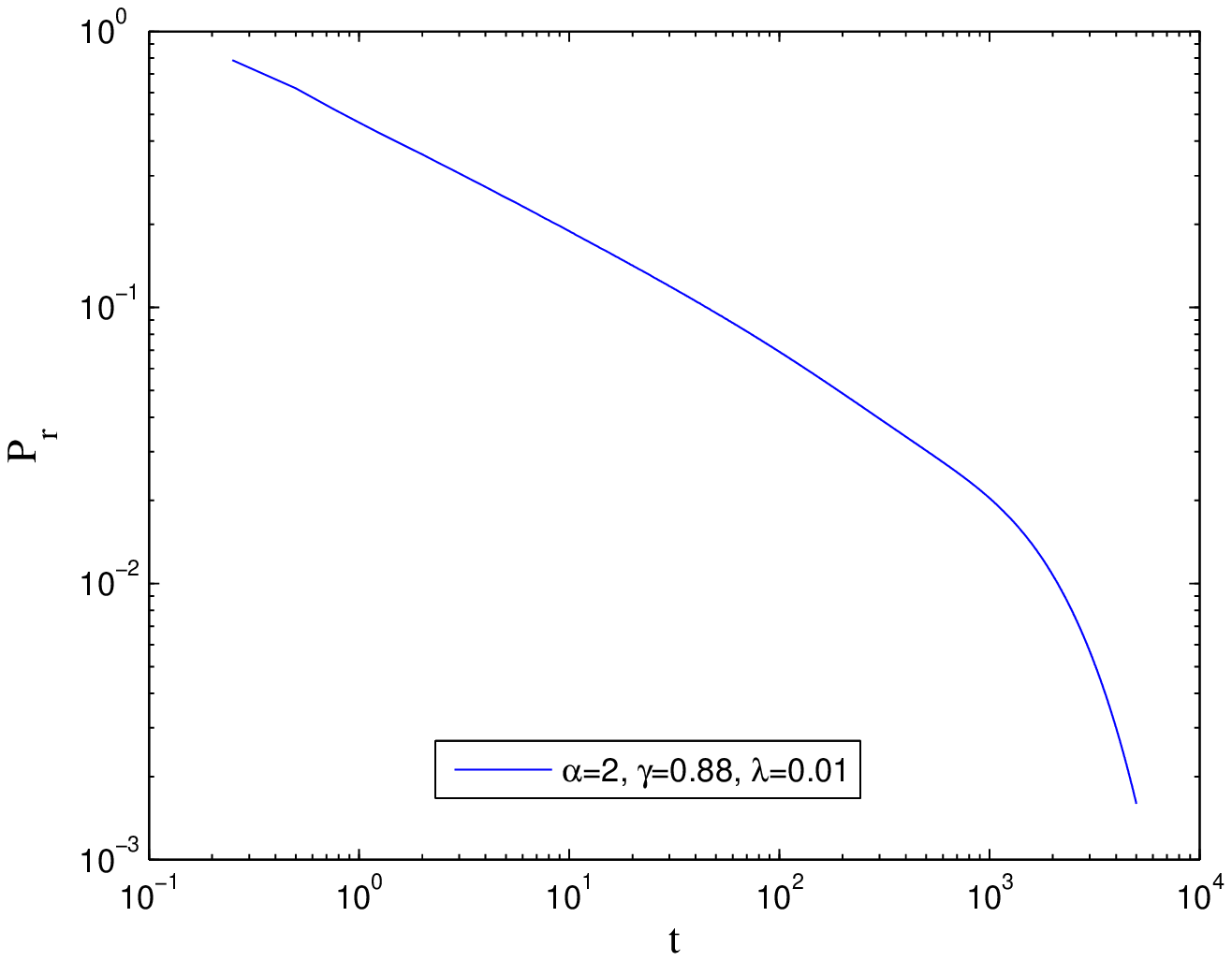}
    \end{minipage}
    \begin{minipage}[t]{0.50\linewidth}
    \includegraphics[scale=0.45]{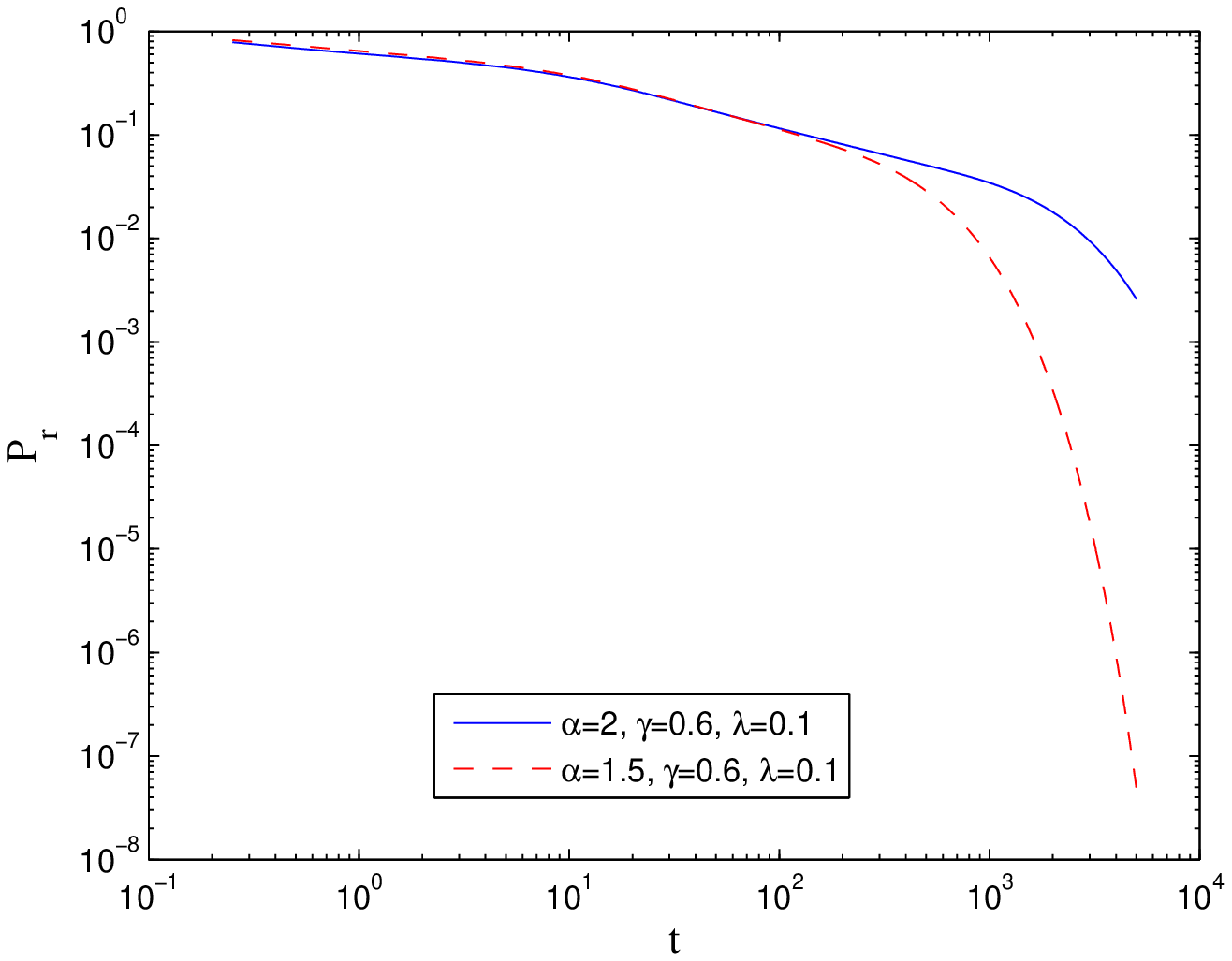}
    \end{minipage}
    \begin{minipage}[t]{0.50\linewidth}
    \includegraphics[scale=0.45]{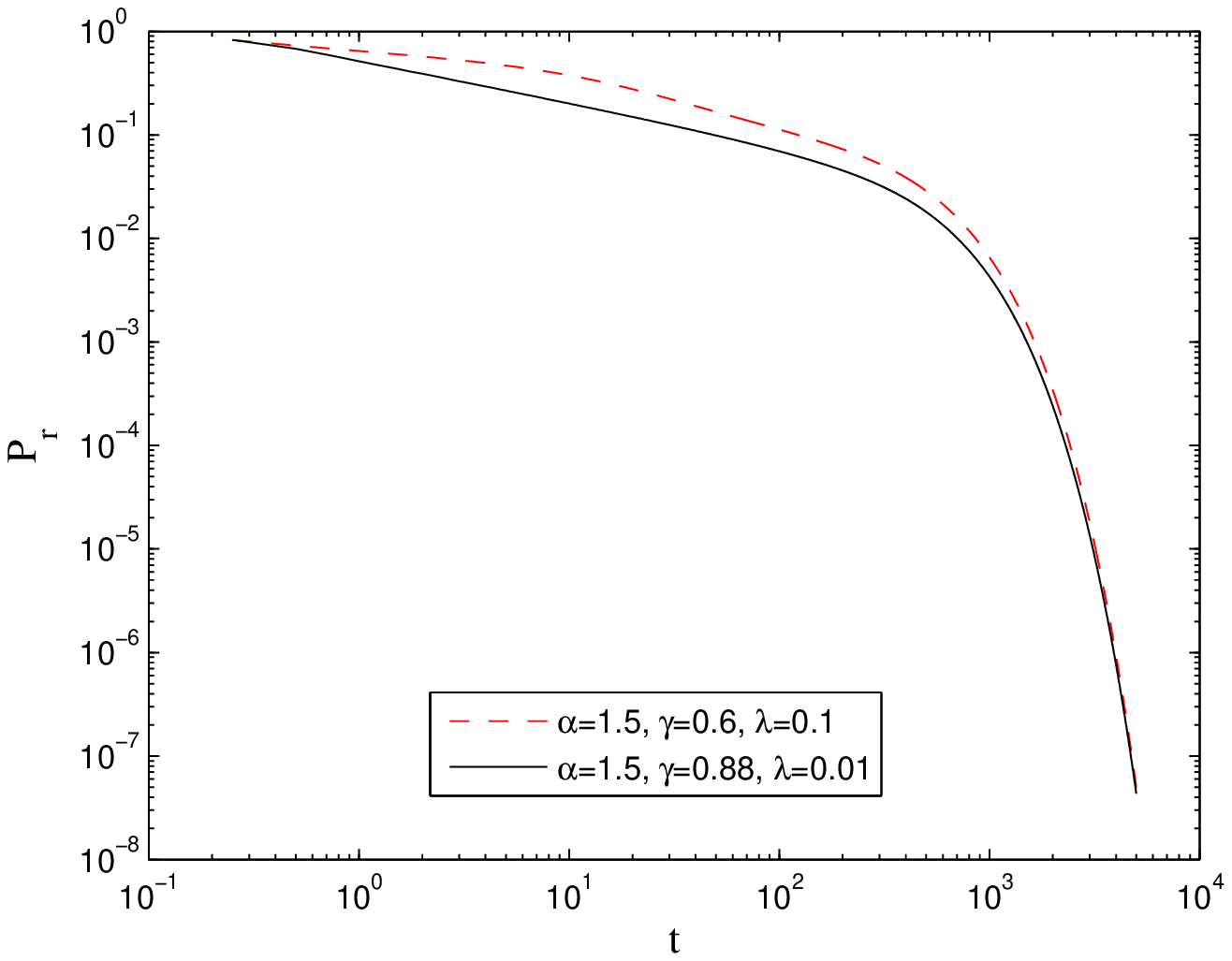}
    \end{minipage}
\caption{Behaviors of $P_r\{ t_f>t\}$ with $c_0=0.05$.}
\end{figure}
Fig. 4.1 shows the simulation results for the probability of the first passage time.  The top-left and top-right ones are for the probabilities of the first passage time of the anomalous dynamics with Gaussian steplength distribution and tempered power-law waiting time distribution; and the simulation results confirm the ones given in \cite{Wu:16}, where the broken curves are used for indicating the slopes. The bottom-left and bottom-right ones also show the probabilities of the first passage time of the anomalous dynamics with power-law steplength distribution.

\section{Conclusions}
With the research deepening of the anomalous dynamics, instead of just focusing on the probability distribution of the position of particles at time $t$, we discuss the distribution of the functionals which involve the entire trajectories. The goal of this paper is to provide the efficient numerical schemes to solve the recently derived time tempered Feynman-Kac equation, which describes the functional distributions of the trajectories of the particles with the tempered power-law waiting time distribution. The high order discretizations of the newly introduced operators are presented and the schemes of the model are designed with detailed stability and convergence proofs. The effectiveness of the algorithms is carefully checked, including simulating the probabilities of the first passage time, being also a specific application of the model. In fact, the second order convergence is also got for solving the original (without the artificial source term) `physical' equation.

\bibliographystyle{amsplain}

\end{document}